\documentclass[reqno]{amsart}

\usepackage{amsrefs}
\usepackage{hyperref}
\usepackage{enumerate,amssymb}

\newtheorem{theorem}{Theorem}
\newtheorem{lemma}[theorem]{Lemma}
\newtheorem{proposition}[theorem]{Proposition}
\newtheorem{corollary}[theorem]{Corollary}

\newtheorem{remark}[theorem]{Remark}
\theoremstyle{exercise}

\theoremstyle{definition}
\newtheorem{definition}[theorem]{Definition}

\newtheorem{conjecture}[theorem]{Conjecture}

\usepackage{graphicx}
\usepackage{pstricks, enumerate}

\numberwithin{equation}{section}
\numberwithin{theorem}{section}

\newenvironment{proofof}[1]{\noindent{\bf Proof of #1.}\hspace*{1em}}{\qed\bigskip}

\newcommand{\intav}[1]{\mathchoice {\mathop{\vrule width 6pt height 3 pt depth  -2.5pt
\kern -8pt \intop}\nolimits_{\kern -6pt#1}} {\mathop{\vrule width
5pt height 3  pt depth -2.6pt \kern -6pt \intop}\nolimits_{#1}}
{\mathop{\vrule width 5pt height 3 pt depth -2.6pt \kern -6pt
\intop}\nolimits_{#1}} {\mathop{\vrule width 5pt height 3 pt depth
-2.6pt \kern -6pt \intop}\nolimits_{#1}}}

\newcommand{\intavl}[1]{\mathchoice {\mathop{\vrule width 6pt height 3 pt depth  -2.5pt
\kern -8pt \intop}\limits_{\kern -6pt#1}} {\mathop{\vrule width 5pt
height 3  pt depth -2.6pt \kern -6pt \intop}\nolimits_{#1}}
{\mathop{\vrule width 5pt height 3 pt depth -2.6pt \kern -6pt
\intop}\nolimits_{#1}} {\mathop{\vrule width 5pt height 3 pt depth
-2.6pt \kern -6pt \intop}\nolimits_{#1}}}

\newcommand{\R}{\mathbb{R}}

\newcommand{\eps}{\epsilon}

\renewcommand{\P}[1]{{\mathbb{P}}\left[{#1}\right]}

\newcommand{\EE}[2]{{\mathbb{E}}\left[{#1}|{#2}\right]}

\begin{document}

\title[Generalized P\'{o}lya's Urn: convergence at linearity]{A generalized P\'{o}lya's Urn with graph based interactions: convergence at linearity}

\author[Jun Chen]{Jun Chen}
\address{Division of the Humanities and Social Sciences, California Institute of Technology, Pasadena, CA 91125.}
\email{chenjun851009@gmail.com}

\author[Cyrille Lucas]{Cyrille Lucas}
\address{Weizmann Institute of Science, Israel and Universit\'{e} Paris Diderot, France.}
\email{lucas@math.univ-paris-diderot.fr}




\date{\today}

\begin{abstract}
We consider a special case of the generalized P\'{o}lya's urn model introduced in \cite{benjamini2012generalized}. Given a finite connected
graph $G$, place a bin at each vertex. Two bins are called a pair if they
share an edge of $G$. At discrete times, a ball is added to each pair of bins.
In a pair of bins, one of the bins gets the ball with probability proportional to its current
number of balls. A question of essential interest for the model is to understand the limiting behavior of the proportion of balls in the bins
for different graphs $G$. In this paper, we present two results regarding this question.
If $G$ is not balanced-bipartite, we prove that the proportion of balls converges to some deterministic point $v=v(G)$ almost
surely. If $G$ is regular bipartite, we prove that the proportion of balls converges to a point in some explicit interval almost surely. The question of convergence remains open in the case when $G$ is non-regular balanced-bipartite (see Remark \ref{final_remark}).
\end{abstract}

\maketitle

\section{Introduction and statement of results}
As a special case of the generalized P\'{o}lya's urn model introduced
in \cite{benjamini2012generalized}, the model with linear
reinforcement is defined as follows.
Let $G=(V,E)$ be a finite connected graph with $V=[m]=\{1,\ldots,m\}$ and $|E|=N$, and assume
that on each vertex $i$ there is a bin initially with $B_i(0)\ge 1$ balls. Consider the random
process of adding $N$ balls to these bins at each step, according to the following
law: if the numbers of balls after step $n-1$ are $B_1(n-1),\ldots,B_m(n-1)$, step $n$ consists
of adding, for each edge $\{i,j\}\in E$, one ball either to $i$ or to $j$, and the probability that the ball
is added to $i$ is
\begin{align}\label{transition probability}
\P{i\text{ is chosen among }\{i,j\}\text{ at step }n}=\dfrac{B_i(n-1)}{B_i(n-1)+B_j(n-1)}\,\cdot
\end{align}
Let $N_0=\sum_{i=1}^m B_i(0)$ denote the initial total number of balls, and let
\begin{align}\label{definition model}
x_i(n)=\dfrac{B_i(n)}{N_0+nN}\,,\ \ \ i\in[m],
\end{align}
be the proportion of balls at vertex $i$ after step $n$. 
Let $x(n)=(x_1(n),\ldots,x_m(n))$. We are interested in the
limiting behavior of $x(n)$ for different graphs $G$.

We call $G$ {\it balanced-bipartite} if there is a bipartition $V=A\cup B$ with $\#A=\#B$.
In \cite{benjamini2012generalized}, the authors proved that when $G$ is not balanced-bipartite, the limit of $x(n)$ exists, and it can
only take finitely many possible values.
Here, we improve this result and prove that almost surely the limit of $x(n)$ is in fact one deterministic point,
thus confirming the conjecture in Section 11 of \cite{benjamini2012generalized}.

\begin{theorem}  \label{not_balanced_bipartite}
Let $G$ be a finite, connected, not balanced-bipartite graph. Then there exists a deterministic point $v=v(G)$
such that $x(n)$ converges to $v$ almost surely.
\end{theorem}
In the proof of this theorem, we will give a characterization of $v(G)$, that enables us to explicitly compute its value for some graphs,
like regular nonbipartite graphs, star graphs, and other small graphs. Hence, Theorem \ref{not_balanced_bipartite} will imply Theorems 1.1(a)
and 1.5(a) of \cite{benjamini2012generalized}.

When $G$ is regular bipartite, the authors in \cite{benjamini2012generalized} proved that the limit set of $x(n)$ is contained in
$\Omega$ where $\Omega$ is the subset of the $(m-1)$-dimensional closed simplex defined as follows:
if $V=A\cup B$ is the bipartition of $G$, then
\begin{align}\label{definition Omega}
\Omega=\{(x_1,\ldots,x_m):\exists\,p,q\ge 0,p+q=2/m,\text{ s.t. }x_i=p\text{ on }A,x_i=q\text{ on }B\}.
\end{align}
Nevertheless, the question whether $x(n)$ has a limit was left open (see Problem 11.2 in \cite{benjamini2012generalized}).
The following theorem provides the answer to this question.

\begin{theorem}  \label{bipartite}
Let $G$ be a finite, regular and bipartite graph. Then $x(n)$ almost surely converges to a point in $\Omega$.
\end{theorem}

The question of the distribution of this random limit in $\Omega$ is left open.

The main technique used in \cite{benjamini2012generalized} is the dynamical system approach (see e.g. \cite{benaim1996dynamical,benaim1999dynamics}),
by which one can analyze the limiting behavior of
$x(n)$ via an approximating ordinary differential equation (ODE). Under some conditions on $x(n)$ and on the ODE,
it was shown that the limit set of $x(n)$ is contained in the \emph{equilibria set} of the ODE. Depending on $G$, the equilibria set can be
either finite or infinite. By a probabilistic argument, the authors in \cite{benjamini2012generalized} also proved that $x(n)$ has
probability zero to
converge to an unstable equilibrium (see Definition \ref{definition-unstable-equilibrium}).

Our results and proofs in this paper are continuation of those in \cite{benjamini2012generalized}.
To prove Theorem \ref{not_balanced_bipartite}, the main work is to prove the uniqueness of a non-unstable equilibrium for
any not balanced-bipartite $G$. The difficulty is that for a general graph, there is no explicit formula for the
equilibria and hence  it is impossible to calculate eigenvalues of the jacobian matrix at equilibria. We overcome
this difficulty by constructing a Lyapunov function. To prove
Theorem \ref{bipartite}, one main difficulty is that the limit set $\Omega$ attracts exponentially in the interior, but not
at its two endpoints. Thus one cannot directly apply the theorem proved in \cite{benaim1996asymptotic} for dealing with the case
where there is a uniform exponential attractor. Then our strategy to prove the convergence in Theorem \ref{bipartite} is to treat 
the convergence to the two endpoints of $\Omega$ and to its interior separately. More precisely, we will prove that the random process 
(\emph{interpolated process}) has to converge to some point in the
interior of $\Omega$ if it does not converge to the endpoints of $\Omega$. The proof uses ideas similar to shadowing techniques
\cite{benaim1996asymptotic,duflo1996algorithmes,schreiber1997expansion}. 
One main reason for our technique to work is due to the special structure of $\Omega$, which is a segment of equilibria that loses exponential attraction 
only at its two endpoints. Naturally, our technique can be applied to a setting where a segment of equilibria attracts exponentially everywhere but not
at finitely many points. 

The organization of this paper is as follows. In Section \ref{the_dynamical_approach}, we do some preparation work for the later proofs: we describe the dynamical system approach in our setting and cite the necessary results from \cite{benjamini2012generalized}. In Sections \ref{section_not_balanced_bipartite} and \ref{section_bipartite}, we prove Theorem \ref{not_balanced_bipartite} and \ref{bipartite} respectively. In Section \ref{balanced_not_regular}, we discuss the model on non-regular balanced-bipartite graphs.

\section{Some results from \cite{benjamini2012generalized}} \label{the_dynamical_approach}

We will first describe the evolution of the model in a way that highlights the underlying deterministic ODE.
Let $\mathcal F_n=\sigma(x(i):0\le i \le n)$ be the filtration generated by $x(i)$ up to step $n$. Then we have the following lemma, which was proved
in Sections 2 and 3 of \cite{benjamini2012generalized}.

\begin{lemma} \label{F_is_the_ODE}
The evolution of $\{x(n)\}_{n\ge 0}$ follows a recursive equation of the form
\begin{align}\label{definition system}
x(n+1)-x(n)=\gamma_n\left[F(x(n))+u(n)\right],
\end{align}
where $F:\R^m \to \R^m$ is a deterministic map, $u(n)$ is a random sequence of vectors with zero conditional mean ($\mathbb E(u(n)|\mathcal F_n)=0$) and $\gamma_n$
is a normalizing factor with $\gamma_n=O(1/n)$.
\end{lemma}
\begin{proof}

Recall that $x_i(n)$ is the fraction of the total number of balls contained in the $i$-th bin at time $n$:
\begin{align}
x_i(n)=\dfrac{B_i(n)}{N_0+nN}\,,\ \ \ i\in[m].
\end{align}
Let $\delta_{i\leftarrow j}(n+1)$ be the indicator of the event that the new ball added on the edge $\{i,j\}$ at step $n+1$ is added to the $i$-th bin. By the definition of the process, we have 
\begin{align}\label{transition probability 2}
\EE{\delta_{i\leftarrow j}(n+1)}{\mathcal F_n}
=\dfrac{x_i(n)}{x_i(n)+x_j(n)}\,\cdot
\end{align}
Now observe that
\begin{eqnarray*}
x_i(n+1)-x_i(n)&=&\dfrac{B_i(n)+\sum_{j\sim i}\delta_{i\leftarrow j}(n+1)}{N_0+(n+1)N}-\dfrac{B_i(n)}{N_0+nN}\\
               &&\\
               &=&\dfrac{-Nx_i(n)+\sum_{j\sim i}\delta_{i\leftarrow j}(n+1)}{N_0+(n+1)N}\\
               &=& \frac{1}{N_0/N+ (n+1)}\left(-x_i(n)+ \frac{1}{N}\sum_{j\sim i} \frac{x_i(n)}{x_i(n)+x_j(n)}   \right)  \nonumber \\
                & &  + \frac{1}{N_0/N+ (n+1)}\cdot \frac{1}{N}\sum_{j\sim i} \left(\delta_{i\leftarrow j}(n+1)-\frac{x_i(n)}{x_i(n)+x_j(n)} \right) \nonumber\cdot
\end{eqnarray*}
Let
\begin{align}\label{definition gamma_n}
\gamma_n=\dfrac{1}{\frac{N_0}{N}+(n+1)}\cdot
\end{align}
Thus, defining the sequence of random vectors $u(n)=(u_i(n))_{i\in [m]} \subset\R^m$ by
\begin{align}\label{definition u_n}
u_i(n)=\frac{1}{N}\sum_{j\sim i} \left(\delta_{i\leftarrow j}(n+1)-\frac{x_i(n)}{x_i(n)+x_j(n)} \right)
\end{align}
and $F=(F_1,\ldots,F_m)$ by
\begin{align}\label{definition F}
F_i(x_1,\ldots,x_m)=-x_i+\dfrac{1}{N}\sum_{j\sim i}\dfrac{x_i}{x_i+x_j}\,,
\end{align}
our random process takes the form
\begin{align}\label{definition system}
x(n+1)-x(n)=\gamma_n\left[F(x(n))+u(n)\right],
\end{align}
where $\mathbb{E}(u(n)|\mathcal F_n)=0$, which concludes the proof of Lemma \ref{F_is_the_ODE}.
\end{proof}

Following a limit set theorem (see e.g. \cite{benaim1996dynamical}), we can analyze the limiting behavior of $x(n)$
by considering its underlying ODE
$dv/dt=F(v), \, v\in \R^m$:
\begin{eqnarray}\label{definition ODE}
\left\{
\begin{array}{rcl}
\dfrac{dv_1(t)}{dt}&=&-v_1(t)+\dfrac{1}{N} \displaystyle\sum_{j\sim 1}\frac{v_1(t)}{v_1(t)+v_j(t)}    \\
&\vdots&\\
\dfrac{dv_m(t)}{dt}&=&-v_m(t)+\dfrac{1}{N} \displaystyle\sum_{j\sim m}\frac{v_m(t)}{v_m(t)+v_j(t)}\cdot
\end{array}\right.
\end{eqnarray}
Let us specify the domain of the vector field $F$. Fix $c<1/N$,
and let $\Delta$ be the set of $m$-tuples $(x_1,\ldots,x_m)\in\mathbb R^m$ such that:
\begin{enumerate}[(1)]
\item $x_i\ge 0$ and $\sum_{i=1}^m x_i=1$, and
\item $x_i+x_j\ge c$ for all $\{i,j\}\in E$.
\end{enumerate}
We equip $\Delta$ with the distance $d$ induced by the $L^1$ norm in $\R^m$. Note that $\Delta$ is positively invariant
(see Lemma 3.4 in \cite{benjamini2012generalized} for a detailed proof), and that the restriction of $F$ to $\Delta$
is  Lipschitz. A point $x\in \Delta$ is called an {\it equilibrium} if $F(x)=0$.
Let $\Lambda$ be the equilibria set of $F$ in $\Delta$.
The following result gives the relation between the limit set of $x(n)$ and $\Lambda$.

\begin{proposition}{\cite[Theorem 3.3]{benjamini2012generalized}} \label{limit_set_theorem}
The limit set of $\{x(n)\}_{n\ge 0}$ is a connected subset of $\Lambda$ almost surely.
\end{proposition}

For the sake of completeness, we sketch the proof of this proposition. It requires the construction of a
Lyapunov function. Let $U\subset \R^m$ be a closed set and $F: U\to \R^m$
be a continuous vector field with unique integral curves.

\begin{definition}[Lyapunov function]\label{definition strict lyapunov function}
A {\it (strict) Lyapunov function for $W\subset U$} is a continuous map $L:U\to\R$ which is (strictly)
monotone along any integral curve of $F$ in $U \setminus W$.
\end{definition}

\begin{proofof}{Proposition \ref{limit_set_theorem}}
We refer the reader to Section 3 of \cite{benjamini2012generalized} for a detailed proof.
We will use the limit set theorem stated therein, which requires
the following conditions:
\begin{enumerate}[(i)]
\item for any $T>0$,
\begin{align*}
\lim_{n\to\infty}\left(\sup_{\{k:0\le\tau_k-\tau_n\le T\}}\left\Vert\sum_{i=n}^{k-1}\gamma_i u(i)\right\Vert\right)=0 \, \text{ a.s.}
\end{align*}
where $\tau_n=\sum_{i=0}^{n-1}\gamma_i$, and
\item $F$ admits a strict Lyapunov function $L$ for $\Lambda$.
\end{enumerate}
We remark that (i) controls the noise perturbation between the random process $x(n)$ and its associated ODE, and
(ii) guarantees the convergence of the ODE to its equilibria.

For (i), let $M_n=\sum_{i=0}^n\gamma_i u(i)$. $\{M_n\}_{n\ge 0}$ is a martingale with bounded quadratic variation, hence it
converges almost surely to a finite random vector
(see e.g. Theorem 5.4.9 of \cite{durrett2010probability}). In particular,
it is a Cauchy sequence and so (i) holds almost surely.

For (ii), let $L:\Delta\rightarrow\R$ be given by
\begin{align}\label{definition L}
L(v_1,\ldots,v_m)=-\sum_{i=1}^m v_i+\dfrac{1}{ N}\sum_{\{i,j\}\in E}\log{(v_i+v_j)}.
\end{align}
Thus
\begin{align}\label{expression of velocity}
\dfrac{dv_i}{dt}=v_i\left(-1+\dfrac{1}{N}\sum_{i\sim j}\dfrac{1}{v_i+v_j}\right)
=v_i\dfrac{\partial L}{\partial v_i}\cdot
\end{align}
If $v=(v_1(t),\ldots,v_m(t))$, $t\ge 0$, is an integral curve of $F$, then (\ref{expression of velocity}) implies
\begin{eqnarray*}
\dfrac{d}{dt}(L\circ v)=\sum_{i=1}^m\dfrac{\partial L}{\partial v_i}\dfrac{dv_i}{dt}
=\sum_{i=1}^m v_i\left(\dfrac{\partial L}{\partial v_i}\right)^2\ge 0.
\end{eqnarray*}
In particular, the last expression is zero if and only if $v_i\left(\frac{\partial L}{\partial v_i}\right)^2=0$ for all
$i\in[m]$, which is equivalent to $F(v)=0$. Hence, $L$ is a strict Lyapunov function for $\Lambda$.

The rest of the proof is a straightforward application of the limit set theorem.
\end{proofof}

Define a face $\Delta_S$ of $\Delta$ as its subset such that $v_i=0$ if and only if $i\notin S\subset [m]$. Let $L|_{\Delta_S}$
be the restriction of $L$ to $\Delta_S$. Since an equilibrium $v$
satisfies  $v_i(\partial L/\partial v_i)=0$ for any $i\in [m]$, we can decompose the equilibria set $\Lambda$ into the union
of the sets of critical points of $L|_{\Delta_S}$ over all faces $\Delta_S$.

When $G$ is not balanced-bipartite, $L$ is strictly concave (see Corollary 1.3 in \cite{benjamini2012generalized}). So for any face $\Delta_S$, $L|_{\Delta_S}$
is strictly concave, and hence has at most one critical point. Therefore, $\Lambda$ is finite. Then it immediately follows
from Proposition \ref{limit_set_theorem} that the limit of $x(n)$ exists in this case. We have the corollary below.

\begin{corollary}{\cite[Corollary 1.3]{benjamini2012generalized}}  \label{finite_limit_points}
Let $G$ be a finite, connected, not balanced-bipartite graph.
Then $\Lambda$ is finite and $x(n)$ converges to an element of $\Lambda$ almost surely.
\end{corollary}

After proving that the limit set of $x(n)$ is contained in $\Lambda$ in Proposition \ref{limit_set_theorem}, we want to understand
which equilibrium $x(n)$ can actually converge to. First we give the following definition.

\begin{definition}[Unstable/non-unstable equilibrium]\label{definition-unstable-equilibrium}
An equilibrium $x$ is called {\it unstable} if
at least one of the eigenvalues of $JF(x)$, the jacobian matrix of $F$ at $x$, has positive real part. Otherwise, we 
call it {\it non-unstable}.
\end{definition}

The following lemma rules out the possibility that $x(n)$ converges to an unstable equilibrium.

\begin{lemma}{}\label{zero probability unstable}
Let $G$ be a finite and connected graph. Let
$v$ be an unstable equilibrium. Then
\begin{align}\label{equation non-convergence unstable}
\P{\lim_{n\to \infty}x(n)=v}=0.
\end{align}
\end{lemma}

The proof of Lemma \ref{zero probability unstable} follows from Lemma 5.2 in \cite{benjamini2012generalized} and the characterization of an unstable equilibrium as shown in the following lemma.

\begin{lemma} \label{equilibrium_v_Lyapunov}
An equilibrium $v$ is unstable if and only if there exists some coordinate $i\in [m]$ with
$v_i=0$ and  $\partial L/\partial v_i>0$.
\end{lemma}

Lemma \ref{equilibrium_v_Lyapunov} was proved in Section 5 of \cite{benjamini2012generalized}. For the sake of completeness, we give its proof here.

\begin{proof}
We look at the
jacobian matrix $JF(v)$:
$$
\frac{\partial F_i}{\partial v_j}=\left\{\begin{array}{ll}
v_i\dfrac{\partial ^2L}{\partial v_i\partial v_j}&\text{ if }i\sim j,\\
&\\
\dfrac{\partial L}{\partial v_i}+v_i\dfrac{\partial ^2L}{\partial v_i^2} &\text{if }i= j, \\
&\\
0&\text{otherwise}.\\
\end{array}\right.
$$
Without loss of generality, assume that $v_i=0$ iff $1\le i\le k$ ($k$ can be zero). Thus
\begin{align}\label{definition jacobian}
JF(v)=\left[
\begin{array}{cc}
A & 0\\
C & B \\
\end{array}
\right]
\end{align}
where $A$ is a $k\times k$ diagonal matrix with $a_{ii}=\partial L/\partial v_i$, $i\in[k]$.
The spectrum of $JF(v)$ is the union of the spectra of $A$ and $B$. With respect to the inner
product $(x,y)=\sum_{i=k+1}^m x_iy_i/v_i$, $B$ is self-adjoint and negative semidefinite
(by the concavity of $L$), hence the eigenvalues of $B$ are real and nonpositive.
Therefore, $JF(v)$ has at least one real positive eigenvalue if and only if at least one of the $a_{ii}$ is positive.
\end{proof}

Let $w=(w_1,\ldots,w_m)$ be a non-unstable equilibrium.
Let $P=\{i\in [m]: w_i>0\}$ and $Z=\{i\in [m]: w_i=0 \}=[m]\setminus P$ denote the coordinates of $w$ with strictly positive and
zero values respectively. Notice that $Z$ can be empty. By the definition of equilibrium and (\ref{expression of velocity}),
if $i \in P$ then $\left. \partial L/\partial v_i \right|_{w}=0$. By Lemma  \ref{equilibrium_v_Lyapunov}, if $i \in Z$ then $\left. \partial L/\partial v_i \right|_{w}\le 0$. Hence, $w$ is non-unstable if and only if it satisfies
\begin{equation} \label{zero_value_set}
\left. \frac{\partial L}{\partial v_i} \right|_{w}\le 0, \,  \forall i\in Z; \, \left. \frac{\partial L}{\partial v_i} \right|_{w}=0, \,  \forall i\in P.
\end{equation}
It can also be seen from these conditions that only boundary equilibria can be unstable.

\section{Not balanced-bipartite graphs: Proof of Theorem \ref{not_balanced_bipartite}} \label{section_not_balanced_bipartite}

By Corollary \ref{finite_limit_points}, if $G$ is not balanced-bipartite,
then the limit of $x(n)$ exists almost surely and is contained in $\Lambda$, i.e.
\begin{equation}  \label{total probability 1}
\sum_{v\in \Lambda} \mathbb P \left[\lim_{n\to \infty} x(n)=v \right] =1.
\end{equation}
By Lemma \ref{zero probability unstable}, the probability that $x(n)$ converges to an unstable equilibrium of $F$ is zero.
Then there exists at least one non-unstable equilibrium.
To prove Theorem \ref{not_balanced_bipartite}, it suffices to prove its uniqueness, which is given
by the following lemma.
\begin{lemma}  \label{unique_stable_equilibrium}
Let $G$ be a finite, connected, not balanced-bipartite graph. Then all but exactly one equilibrium are unstable.
\end{lemma}

\begin{proof}
The proof uses properties of the vector field $F$ only (the random sequence $\{x(n)\}_{n\geq 0}$ plays no role).
Let $w=(w_1,\ldots,w_m)$ be a non-unstable equilibrium.
We claim that for any $v^0 \in{\rm int}(\Delta)$, the orbit $\{v(t) \}_{t\ge 0}$ with $v(0)=v^0$ converges to $w$. Clearly,
this implies the uniqueness of $w$.

Recall the definition of $P$ in the end of the previous section. We prove the claim by constructing a Lyapunov function $H: \{v\in \Delta: v_i>0, \, \forall i\in P\}\to \R$,
\begin{equation}  \label{Lyapunov_function_uniqueness}
H(v)=\sum_{i\in P} w_i \log{v_i}.
\end{equation}
It is worth noting that ${\rm int}(\Delta)\subset \{v\in \Delta: v_i>0, \, \forall i\in P\}$ and that $H(v)\le 0$ in $\Delta$.  Set $c^0=\sum_{i\in P} w_i \log{v_i(0)}$, and
consider $H^{-1}[c^0, 0]=\{v\in \Delta: H(v)\ge c^0 \}$.
Observe that  there exists some small $\widetilde c>0$ such that
\[
H^{-1}[c^0, 0] \subset \{u\in \Delta: u_i\ge \widetilde c, i\in P  \}.
\]
Hence if the orbit $v(t)$ is located in the set $H^{-1}[c^0, 0]$, it is legitimate to take the derivative of $H(v)$ along it:
\begin{eqnarray*}
\dfrac{d H(v(t))}{d t}&=&\dfrac{d}{dt}\left(\displaystyle \sum_{i\in P} w_i \log{v_i} \right)  \nonumber \\
                              &=& \sum_{i\in P} w_i \dfrac{1}{v_i}\left(-v_i+\dfrac{1}{N} \sum_{j\sim i} \frac{v_i}{v_i+v_j}\right)  \nonumber \\
                              &=& \sum_{i\in P} w_i \left(-1+\dfrac{1}{N} \sum_{j\sim i} \frac{1}{v_i+v_j}\right)\cdot
\end{eqnarray*}
Since $w_i=0$ for $i\in Z$, it follows from above that
\begin{eqnarray} \label{general, lyapunov derivative}
\dfrac{d H(v(t))}{d t}&=& \sum_{i=1}^m w_i \left(-1+\dfrac{1}{N} \sum_{j\sim i} \frac{1}{v_i+v_j}\right)  \nonumber \\
                              &=& -1+ \dfrac{1}{N}\sum_{i=1}^m w_i \left( \sum_{j\sim i} \frac{1}{v_i+v_j}\right)  \nonumber \\
                              &=&-1+\dfrac{1}{N}\sum_{\{i,j\}\in E} \dfrac{w_i+w_j}{v_i+v_j}\cdot
\end{eqnarray}
By Lemma \ref{deterministic_inequality} below, (\ref{general, lyapunov derivative}) is non-negative with equality if and only if $v=w$.
This immediately implies that $H^{-1}[c^0, 0]$ is positively invariant, and that $v(t)$ converges to $w$. This completes the proof of the claim.

\end{proof}

In the proof of Lemma \ref{unique_stable_equilibrium}, we made use of the following lemma.

\begin{lemma}  \label{deterministic_inequality}
Let $G$ be a finite, connected, not balanced-bipartite graph.
If $w$ is a non-unstable equilibrium, then
\[
f(v_1,\ldots,v_m)= \sum_{\{i,j\}\in E} \dfrac{w_i+w_j}{v_i+v_j}\ge N, \, \forall (v_1,\ldots,v_m)\in \Delta,
\]
with equality if and only if $v=w$.
\end{lemma}

\begin{proof}

The proof of Lemma \ref{deterministic_inequality} follows from the following two claims:
\begin{enumerate}[(a)]
\item $w$ is a strict local minimum of $f(\cdot)$ in $\Delta$;
\item $f(\cdot)$ is strictly convex in $\Delta$.
\end{enumerate}

Let's prove (a).

Let $\eps=(\eps_1,\ldots,\eps_m)$. Observe that we can write any point in a neighborhood of $w$ as $w^\eps=(w_1+\eps_1, \ldots, w_m+\eps_m)$ with $\sum_{i=1}^m \eps_i=0$.
By the following elementary inequality
\[
\frac{x}{x+\eps}-1 \ge -\frac{\eps}{x}, \, \, \forall x>0, \, \eps> -x,
\]
we have
\begin{eqnarray}  \label{local minimum_general}
f(w^\eps)- f(w) &=& \sum_{\{i,j\}\in E} \left[\dfrac{w_i+w_j}{w_i+\eps_i+w_j+\eps_j}-1\right] \nonumber \\
&\ge & -\sum_{\{i,j\}\in E} \frac{\eps_i+\eps_j}{w_i+w_j}  \nonumber \\
&= & -\sum_{i=1}^m \eps_i \sum_{j\sim i} \frac{1}{w_i+w_j}\cdot
\end{eqnarray}
Since $w^\eps \in \Delta$, we must have $\eps_i\ge 0$ for any $i\in Z$. By (\ref{zero_value_set}),
\begin{eqnarray*}
\sum_{i=1}^m \eps_i \sum_{j\sim i} \frac{1}{w_i+w_j} &=&  \sum_{i\in P} \eps_i \sum_{j\sim i} \frac{1}{w_i+w_j} +\sum_{i\in Z} \eps_i \sum_{j\sim i} \frac{1}{w_j} \nonumber \\
                  &=&  \sum_{i\in P} \eps_i\cdot N + \sum_{i\in Z} \eps_i \sum_{j\sim i} \frac{1}{w_j}  \nonumber \\
                  &\le&  \sum_{i\in P} \eps_i\cdot N +\sum_{i\in Z} \eps_i \cdot N  \nonumber \\
                  &= & N\sum_{i=1}^m \eps_i = 0. \nonumber \\
\end{eqnarray*}
Then it follows that $f(w^\eps)- f(w)\ge 0$.

Notice that (\ref{local minimum_general}) has equality if and only if
\begin{equation}  \label{equality epsilon}
\eps_i+\eps_j=0,  \quad \forall \{i,j\}\in E.
\end{equation}
If $G$ is a non-bipartite graph, it has an odd cycle, then (\ref{equality epsilon}) implies $\eps=0$.
If $G$ is bipartite but not balanced-bipartite, (\ref{equality epsilon}) together with $\sum_{i=1}^m \eps_i=0$ implies $\eps=0$.
In both cases, i.e. if $G$ is not balanced-bipartite, then $f(w^\eps)- f(w)> 0$ for all small $\epsilon\not=0$. This completes the proof of (a).

Now we will prove (b).

For any $u, v\in \Delta$ and $0<t<1$, by the convexity of the function $\frac{1}{x}\, (x>0)$,
\begin{eqnarray}  \label{convex function}
f(tu+(1-t)v) &=& \sum_{\{i,j\}\in E} \left[\dfrac{w_i+w_j}{(tu_i+(1-t)v_i)+(tu_j+(1-t)v_j)}\right] \nonumber \\
& \le &  \sum_{\{i,j\}\in E} \left[t\dfrac{w_i+w_j}{u_i+u_j}+(1-t)\dfrac{w_i+w_j}{v_i+v_j}\right]   \nonumber \\
&\le &  tf(u)+(1-t)f(v).
\end{eqnarray}
Notice that (\ref{convex function}) has equality if and only if
\begin{equation}  \label{equality convex}
u_i+u_j=v_i+v_j,  \quad \forall \{i,j\}\in E.
\end{equation}
Set $g(i)=u_i-v_i$. Then (\ref{equality convex}) implies
\begin{equation}  \label{equality convex 2}
g(i)+g(j)=(u_i-v_i)+(u_j-v_j)=0,  \quad \forall \{i,j\}\in E.
\end{equation}
By a similar argument as before, if $G$ is not balanced-bipartite then
\[
f(tu+(1-t)v)- \left(tf(u)+(1-t)f(v)\right)<0,\ \forall\, u\neq v.
\]
This implies that $f(\cdot)$ is strictly convex. We completes the proof of (b).
\end{proof}

By (\ref{total probability 1}) and Lemma \ref{unique_stable_equilibrium}, $x(n)$ then converges almost surely to a {\em unique} non-unstable equilibrium $w$. Hence, Theorem \ref{not_balanced_bipartite} holds with $v(G)=w$, which is characterized by the conditions in (\ref{zero_value_set}).

\section{Regular bipartite graphs: Proof of Theorem \ref{bipartite}}  \label{section_bipartite}

Let $x(t)$ denote the interpolated process of $x(n)$:
\[
x(t)=\sum_{n\ge 0} \left(x(n)+\frac{t-\tau_n}{\gamma_n} (x(n+1)-x(n)) \right)1_{[\tau_n, \tau_{n+1})}(t),
\]
where $\tau_n=\sum_{k=0}^{n-1} \gamma_k$. To prove the convergence of $x(n)$ in Theorem \ref{bipartite}, it suffices to 
prove the convergence of $x(t)$. 

Let $\Phi=\Phi_t(x)$ be the semiflow induced by (\ref{definition F}) where
$t\ge 0$ is the time parameter and $\Phi_0(x)=x$.
Then the following lemma gives a quantitative estimate on how well the interpolated
process can be approximated by the semiflow $\Phi$.

\begin{lemma}{\cite[Proposition 8.3]{benaim1999dynamics}} \label{exponential_approximation}
Almost surely,
\[
\sup_{T>0} \limsup_{t\to \infty} \frac{1}{t} \log \left(\sup_{0\le h \le T} d(x(t+h), \Phi_h(x(t))) \right)\le -1/2.
\]
\end{lemma}

The right-hand side of the inequality above depends on the decrease rate of $\gamma_n$\footnote{More specifically, it
equals $\frac{1}{2}\limsup_{n\to \infty}\frac{\log\gamma_n}{\tau_n}$. If $\gamma_n=O(1/n)$, then $\frac{1}{2}\limsup_{n\to \infty}\frac{\log\gamma_n}{\tau_n}=-\frac{1}{2}$.}.
In \cite{benjamini2012generalized}, the authors proved that when $G$ is regular bipartite, the distance between $x(n)$ and $\Omega$ converges to zero.
\begin{lemma}{\cite[Theorem 1(b)]{benjamini2012generalized}} \label{limit_set_bipartite}
Let $G$ be a finite, regular, connected and bipartite graph, then $\lim_{n\to \infty} d(x(n), \Omega)=0$ almost surely.
\end{lemma}

When $G$ is $r-$regular and bipartite, one can explicitly calculate $JF(v)$, the jacobian matrix of $F$
at a point $v=(p,\ldots,p,q,\ldots,q)\in \Omega$. Let $A$ and $B$ denote
the bipartition of $G$ as before. If we label the vertices of $A$ from $1$ to $m/2$, the vertices of $B$ from $1$ to $m/2$,
and if we let $M=(m_{ij})$ be the $m/2\times m/2$ adjacency matrix of the edges connecting vertices of $A$ to
vertices of $B$ (i.e. $m_{ij}=1$ when the $i$-th vertex of $A$ is adjacent to the $j$-th vertex of $B$),
then $JF(v)$ takes the form
\begin{eqnarray*}
JF(v)=-I+\dfrac{m}{2r}
\left[
\begin{array}{ccc}
rqI&&-pM\\
&&\\
-qM^t&&rpI\\
\end{array}
\right]\cdot
\end{eqnarray*}
Let $l$ be the vector in the tangent space of $\Delta$ with coordinates
\begin{eqnarray*}
l_i&=&\left\{\begin{array}{ll}
1&\text{ if  }i\in A, \\
-1&\text{ if  }i\in B.\\
\end{array}\right.
\end{eqnarray*}
Then it is easy to check that $JF(v)\cdot l=0$ for any $v\in \Omega$. This implies that the jacobian matrix has zero eigenvalue
along the direction of $\Omega$. Let 
$v_{\pm\infty}$ denote the two endpoints of $\Omega$. One can easily see that $JF(v_{\pm\infty})$ has multiple zero eigenvalues. In the interior of $\Omega$, the authors in \cite{benjamini2012generalized} proved that in any direction transverse to $\Omega$, the
eigenvalues have negative real part.

\begin{lemma}{\cite[Lemma 10.1]{benjamini2012generalized}} \label{lemma hyperbolic attractor}
Let $v\in {\rm int}(\Omega)$. Any eigenvalue of $JF(v)$ different from $0$ has negative real part,
and $0$ is a simple eigenvalue of $JF(v)$.
\end{lemma}

Lemma \ref{lemma hyperbolic attractor} says that the interior of $\Omega$ attracts exponentially along any 
direction transverse to $\Omega$. This is a strong 
property about $\Omega$, and it enables us to effectively 
work with the dynamics of the ODE at the interior of $\Omega$.  For a fixed interval $J\subset\Omega$ not containing $v_{\pm\infty}$, and
a small neighborhood $U$ of $J$ in $\Delta$, by Lemma \ref{lemma hyperbolic attractor}, there is a submanifold $\mathcal F_x$ for  each $x\in U$ such that:
\begin{enumerate}[$\bullet$]
\item $\mathcal F_x \pitchfork \Omega$ is one point. We denote this point by $\pi(x)$.
\item The dynamics of the ODE on $\mathcal F_x$ is
exponentially contracting to $\pi(x)$. The speed of convergence
depends on the non-zero eigenvalues of $JF(\pi(x))$.
\end{enumerate}
This follows from the theory of invariant manifolds for normally hyperbolic sets (see Theorem 4.1 of \cite{hirsch1977invariant}).

Thus we have a map $\pi:U\to \Omega$. Notice that $\pi$
is not a projection (it is not even linear), but  $\mathcal F_x$ depends smoothly on $x$. Hence if $U$ is small, then $\pi$ is 2-Lipschitz: 
\begin{equation}  \label{2Lipschitz}
d(\pi(x),\pi(y))\le
2d(x,y),\forall\,x,y\in U.
\end{equation}

Now fix a small parameter $\varepsilon>0$ and reduce $U$, if necessary, so that
\begin{align}\label{equation 1}
U=\{x\in\Delta:\pi(x)\in J\text{ and }d(x,\pi(x))<\varepsilon\}.
\end{align}
Let $c=\max\{{\rm Re}(\lambda):\lambda\not=0\text{ is eigenvalue of
}JF(x), x\in J\}$.
By Lemma \ref{lemma hyperbolic attractor} , $c<0$. Thus there is $K>0$ such that
\begin{align}\label{equation hyperbolicity}
d(\Phi_t(x),\pi(x))\le K e^{ct}d(x,\pi(x)), \forall\,x\in U,\forall\,t\ge 0.
\end{align}

Let $x(t)$ be an orbit that does not converge
to $v_{\pm\infty}$. By Lemma \ref{limit_set_bipartite}, this orbit has an accumulation point in the interior
of $\Omega$. Let $J\subset \Omega$ be an interval containing this point but not $v_{\pm\infty}$, and $U$ as in (\ref{equation 1}).

\begin{lemma} \label{iteration_lemma}
Let $x(t)\in U$. If $t,T$ are large enough, then
\begin{enumerate}[(i)]
\item $d(\pi(x(t+T)),\pi(x(t)))<2e^{-\frac{t}{4}}$.
\item $x(t+T)\in U$.
\end{enumerate}
\end{lemma}

\begin{proof} To simplify the notation, denote $x(t)$ by $x$ and $x(t+T)$ by
$x(T)$.

Let's prove \noindent (i). Since $\pi(\Phi_T(x))=\pi(x)$, and $\pi$ is 2-Lipschitz,
\begin{align*}
d(\pi(x(T)),\pi(x))=d(\pi(x(T)),\pi(\Phi_T(x)))\le
2d(x(T),\Phi_T(x)).
\end{align*}
By Lemma \ref{exponential_approximation}, $d(x(T),\Phi_T(x))\le e^{-\frac{t}{4}}$ for large $t$, therefore 
$d(\pi(x(T)),\pi(x))\le 2e^{-\frac{t}{4}}$ for large $t$.
This proves \noindent (i). Note that \noindent (i) implies that $\pi(x(T))\in J$ for large $t$.

For \noindent (ii), we just need to estimate $d(x(T),\pi(x(T)))$. By the triangular inequality, (\ref{2Lipschitz}) and (\ref{equation hyperbolicity}), we have:
\begin{eqnarray}  \label{distance to Omega}
d(x(T),\pi(x(T)))
&\le&d(x(T),\Phi_T(x))+d(\Phi_T(x),\pi(\Phi_T(x)))+ \nonumber \\
&&d(\pi(\Phi_T(x)),\pi(x(T))) \nonumber \\
&\le&3d(x(T),\Phi_T(x))+d(\Phi_T(x),\pi(x)) \nonumber \\
&\le&3e^{-\frac{t}{4}}+Ke^{cT}d(x,\pi(x)) \nonumber \\
&\le&3e^{-\frac{t}{4}}+Ke^{cT}\varepsilon \nonumber \\
&<&\varepsilon
\end{eqnarray}
whenever $3e^{-\frac{t}{4}}<\frac{\varepsilon}{2}$ and $Ke^{cT}<\frac{1}{2}$.
\end{proof}

Note that Lemma \ref{iteration_lemma}(ii) allows us to iteratively apply Lemma \ref{iteration_lemma} to the points
$x_k:=x(t+kT),k\in \mathbb N$. Hence
$d(\pi(x_{k+1}),\pi(x_k))<2e^{-\frac{t+kT}{4}}$
for all $k\ge 0$. Because $\sum_k e^{-\frac{t+kT}{4}}<\infty$, it follows that $\pi(x_k)$
converges. In the above iterative argument, we implicitly used the fact that $\sum_k e^{-\frac{t+kT}{4}}$ can be made arbitrarily small if $t$ and $T$ are large enough. This fact guarantees that the total drift of $\pi(x_k)$ from $\pi(x)$ is arbitrarily small so that 
$\pi(x_k)\in J$ for all $k\ge 0$, and thus the iterative argument works.

Also note that (\ref{distance to Omega}) holds for all $k\geq 0$:
\begin{equation}  \label{sequence foaliation distance}
 d(x_k,\pi(x_k))\le 3e^{-\frac{t+(k-1)T}{4}}+Ke^{cT}d(x_{k-1},\pi(x_{k-1})).
\end{equation}
Let $\lambda=Ke^{cT}$. Iterating (\ref{sequence foaliation distance}), we get
\begin{eqnarray}
 d(x_k,\pi(x_k))&\le& 3e^{-\frac{t}{4}}\left( e^{-\frac{(k-1)T}{4}}+\lambda e^{-\frac{(k-2)T}{4}} +\cdots+\lambda^{k-1}\right)+\lambda^k d(x,\pi(x))  \nonumber  \\
&\le& 3e^{-\frac{t}{4}}k \left(\max{ \left\{e^{-\frac{T}{4}},\lambda\right\}} \right)^{k-1}+\lambda^k d(x,\pi(x)).  \nonumber  
\end{eqnarray}
When $T$ is large, $\max{ \left\{e^{-\frac{T}{4}},\lambda\right\}} <1$, hence $d(x_k,\pi(x_k)) \to 0$ as $k\to \infty$. 

Let $x_0\in J$ be the limit of $\pi(x_k)$. By the triangular inequality
$$
 d(x_k,x_0)\leq d(x_k,\pi(x_k)) +d(\pi(x_k),x_0).
$$
When $k$ tends to infinity, we have just proved that both  $d(\pi(x_k),x_0)$ and $d(x_k,\pi(x_k))$ go to zero, thus $d(x_k,x_0)$ goes to zero. This proves that $\lim_{k\to \infty} x_k$ exists, with $\lim_{k\to \infty} x_k=x_0\in J$. 

For any $s \in [t+kT, t+(k+1)T)$,  by the triangular inequality and Lemma \ref{exponential_approximation}
\begin{eqnarray*}
d(x(s), x_0) &=& d(x(s), \Phi_{s-(t+kT)}(x_0)) \\
             &\le &d(x(s), \Phi_{s-(t+kT)}(x_k)) +d(\Phi_{s-(t+kT)}(x_k), \Phi_{s-(t+kT)}(x_0)) \\
            &\le &e^{-\frac{t+kT}{4}} + c(T)d(x_k,x_0),
\end{eqnarray*}
where $c(T)>0$ is the supremum of the Lipschitz constants of $\Phi_\delta,\delta\in[0,T]$. Therefore, $\lim_{t\to \infty} x(t)=x_0$. 
This completes the proof of Theorem \ref{bipartite}.

\section{Non-regular balanced-bipartite graphs}\label{balanced_not_regular}

We now discuss non-regular balanced-bipartite graphs. It is the only family of graphs
that we do not have precise information on the convergence of $x(n)$. 

\begin{lemma} \label{BBNR}
Let $G$ be a non-regular balanced-bipartite graph. Then $\Lambda \cap{\rm int}(\Delta)$ is either empty or an interval.
\end{lemma}

\begin{proof}
For a non-regular balanced-bipartite graph, the corresponding ODE can \textit{a priori} have either no or at least one interior equilibrium
in $\Delta$.
Now suppose that the ODE has an interior equilibrium $v$.
Let $V=A\cup B$  be the bipartition of $G$. Then for any $\eta$ with $-\min_{i\in A}v_i<\eta<\min_{i\in B}v_i$,
the points $u^{\eta}=(u^{\eta}_1,\ldots,u^{\eta}_m)$ defined by
\begin{eqnarray}  \label{interval_equilibria}
u^{\eta}_i&=&\left\{\begin{array}{ll}
v_i+\eta&\text{ if  }i\in A, \\
v_i-\eta&\text{ if  }i\in B,\\
\end{array}\right.
\end{eqnarray}
form an interval of interior equilibria.

Furthermore, if $h$ is another interior equilibrium, we will prove that $h$ is contained in this interval.  Recall that
$$L(v)=L(v_1,\ldots,v_m)=-\sum_{i=1}^m v_i+\dfrac{1}{ N}\sum_{\{i,j\}\in E}\log{(v_i+v_j)}.$$
By (\ref{expression of velocity}),
$h$ and $v$ are critical points of $L$.
Since $L$ is concave, $h$ and $v$ are global maxima of $L$ in $\Delta$ and $L(h)=L(v)$.
Then for any $0<c<1$, the following holds:
\begin{equation} \label{concavity_condition}
L(ch+(1-c)v)=cL(h)+(1-c)L(v).
\end{equation}
Since the $\log$ function is strictly concave, (\ref{concavity_condition}) yields $h_i+h_j=v_i+v_j$ for every $\{i,j\}\in E$, i.e.
\begin{equation} \label{equality_condition}
h_i-v_i=-(h_j-v_j)\, ,\ \ \forall\, \{i,j\}\in E.
\end{equation}
Hence, there exists $\eta \in (-\min_{i\in A}v_i,\min_{i\in B}v_i)$, such that
\begin{eqnarray*}
h_i&=&\left\{\begin{array}{ll}
v_i+\eta&\text{ if  }i\in A, \\
v_i-\eta&\text{ if  }i\in B,\\
\end{array}\right.
\end{eqnarray*}
which completes the proof.
\end{proof}

Observe that the proof of Lemma \ref{BBNR} works for any balanced-bipartite graph. Thus, we have proved that for a balanced-bipartite graph, the corresponding $F$ either does not have an interior equilibrium, or has an interval of interior equilibria.

\begin{corollary} \label{no_interior_equilibrium}
Let $G$ be a non-regular balanced-bipartite graph. Assume that $F$ does not have an interior equilibrium, then $\Lambda$ is finite, and $x(n)$ converges to an element of $\Lambda$ almost surely.
\end{corollary}

\begin{proof}
By Proposition \ref{limit_set_theorem}, we just need to prove that $\Lambda$ is finite. Since $\Lambda$ is the union of the sets of critical points of $L|_{\Delta_S}$ over all faces $\Delta_S$ and the total number of faces is finite, it suffices to prove that for each face $\Delta_S$ with $S\neq [m]$, $L|_{\Delta_S}$ is strictly concave. Fix a face $\Delta_S$ with $S\neq [m]$. Let $u,v\in\Delta_S$ and
$c\in(0,1)$. If $L(cu+(1-c)v)=cL(u)+(1-c)L(v)$, then $u_i+u_j=v_i+v_j$ for every $\{i,j\}\in E$, i.e.
\begin{equation}  \label{face_concave}
u_i-v_i=(-1)(u_j-v_j)\, ,\ \ \forall\, \{i,j\}\in E.
\end{equation}
Since $S\neq [m]$ and $u,v\in\Delta_S$, there exists some $i\notin S$ such that $u_i-v_i=0$. Because $G$ is connected, (\ref{face_concave}) implies that $u_i-v_i=0$ for all $i\in [m]$, i.e. $u=v$. This proves that 
$L|_{\Delta_S}$ is strictly concave. We complete the proof of the corollary. 
\end{proof}

Notice that the proof of Corollary \ref{no_interior_equilibrium} is general and only uses the assumption that $F$ does not have an interior equilibrium. We just proved that for any finite connected graph $G$, $L|_{\Delta_S}$ with $S\neq [m]$ is strictly concave, and hence the corresponding $F$ has at most finitely many boundary equilibria. 

For a non-regular balanced-bipartite $G$, if
$F$ does not have an interior equilibrium, by Corollary \ref{no_interior_equilibrium}, we conjecture that there is a unique non-unstable equilibrium of $F$ such that $x(n)$ almost surely converges to it. If $F$ has an interval of interior equilibria, unlike the case of regular bipartite graphs, we are not able to prove a result similar to Lemma \ref{lemma hyperbolic attractor}, and hence not able to prove the convergence of $x(n)$. But we also conjecture that this convergence holds. Combining the results we already proved, we make the conjecture below.
\begin{conjecture}  \label{the_conjecture}
Let $G$ be a finite and connected graph. Then there exists either a point $v(G)$ such that $x(n)$ almost surely converges to $v(G)$
or an interval $\Omega(G)$ such that $x(n)$ almost surely converges to a point in $\Omega(G)$.
\end{conjecture}

\begin{remark} \label{final_remark}
We just learned that the convergence of $x(n)$ for non-regular balanced-bipartite graphs was proved in \cite{lima2014completion}, and hence Conjecture \ref{the_conjecture} was confirmed.
\end{remark}

\section{Acknowledgements}

The authors are thankful to Michel Bena\"im, Itai Benjamini and Pascal Maillard for many enlightening discussions, as well as Ofer Zeitouni for substantial help in the proof of Theorem \ref{bipartite}. During the preparation of this manuscript, J.C. was a student and C.L. was a Postdoctoral Fellow at
the Weizmann Institute of Science. Both authors were supported by the ISF. Finally, the authors would like to thank Yuri Lima for suggesting the simpler proof of Theorem \ref{bipartite} presented here.

\bibliography{convergence_at_linearity_bib}

\end{document}